\documentclass[
final
, nomarks
]{dmtcs-episciences}

\usepackage[utf8]{inputenc}
\usepackage{amsmath,amssymb,amsthm}
\usepackage{enumitem}
\usepackage{subcaption}
\usepackage{tikz}
\newtheorem{theorem}{Theorem}
\newtheorem{corollary}[]{Corollary}
\newtheorem{lemma}[]{Lemma}
\newtheorem{conjecture}[]{Conjecture}
\newtheorem{observation}[]{Observation}
\newtheorem{definition}[]{Definition}

\newtheorem{claim_new}{Claim}
\newtheorem{subclaim}{Subclaim}{\bf}{\it}
\newtheorem{case}{Case}

\newtheorem{subcase}{Case}
\numberwithin{subcase}{case}

\numberwithin{subcase1}{case1}

\newtheorem{case2}{Case}
\newtheorem{subcase2}{Case}
\numberwithin{subcase2}{case2}

\usepackage[round]{natbib}

\author[Nevil Anto, Manu Basavaraju]{Nevil Anto
  \and Manu Basavaraju }

\title{Gallai's Path Decomposition for 2-degenerate Graphs}

\affiliation{
 National Institute of Technology Karnataka, Surathkal, India}
\begin{document}

\publicationdata{vol. 25:1}
{2023}
{16}
{10.46298/dmtcs.10313}
{2022-11-16; 2022-11-16; 2023-05-03}
{2023-05-04}

\maketitle
\begin{abstract}
Gallai's path decomposition conjecture states that if $G$ is a connected graph on $n$ vertices, then the edges of $G$ can be decomposed into at most $\lceil \frac{n }{2} \rceil$ paths. A graph is said to be an {\em odd semi-clique} if it can be obtained from a clique on $2k+1$ vertices by deleting at most $k-1$ edges. Bonamy and Perrett asked if the edges of every connected graph $G$ on $n$ vertices can be decomposed into at most $\lfloor \frac{n}{2} \rfloor$ paths unless $G$ is an odd semi-clique. A graph $G$ is said to be \emph{ 2-degenerate} if every subgraph of $G$ has a vertex of degree at most $2$. In this paper, we prove that the edges of any connected 2-degenerate graph $G$ on $n$ vertices can be decomposed into at most $\lfloor \frac{n }{2} \rfloor$ paths unless $G$ is a triangle.
\end{abstract}
\section{Introduction}\label{sec:intro}

All graphs in this paper are simple, undirected and finite. A \emph{path decomposition} of a graph is a partition of the edge set of the graph into paths. Gallai made the following conjecture on path decomposition:

\begin{conjecture} [Gallai's path decomposition conjecture]\label{conjecture:gallai}
If $G$ is a connected graph on n vertices, then the edges of $G$ can be decomposed into at most $\lceil \frac{n }{2} \rceil$ paths.
\end{conjecture}
\cite{MR0233723} proved that the edges of a connected graph $G$ on $n$ vertices can be decomposed into at most $\lfloor \frac{n }{2} \rfloor$ paths and cycles. An \emph{even subgraph} $I$ of $G$ is defined as a subgraph induced by all the vertices having an even degree in $G$. \cite{MR0233723} proved the conjecture for graphs whose even subgraph consists of a single vertex. \cite{Pyber1996} later proved that the conjecture holds if the even subgraph of a graph is a forest. \cite{fan2005path} proved that the conjecture holds if each block of the even subgraph is a triangle-free graph of maximum degree at most three. \cite{https://doi.org/10.1002/jgt.22647} generalized the result of \cite{fan2005path}. \\ 
\cite{bonamy2019gallai} proved Gallai's conjecture for all graphs with maximum degree at most five. A graph is said to be an {\em odd semi-clique} if it can be obtained from a clique on $2k+1$ vertices by deleting at most $k-1$ edges. \cite{bonamy2019gallai} asked if the edges of every connected graph $G$ on $n$ vertices can be decomposed into at most $\lfloor \frac{n}{2} \rfloor$ paths unless $G$ is an odd semi-clique. \cite{botler2020gallai} proved that the edges of any connected graph $G$ with treewidth at most 3 can be decomposed into at most $\lfloor \frac{n}{2} \rfloor$ paths unless $G$ is $K_3$ or $K_5-e$. \cite{botler2019gallai} proved that the edges of triangle-free planar graphs can be decomposed into at most $\lfloor \frac{n}{2} \rfloor$ paths. More recently, \cite{Chu2021} proved that the edges of any connected graph $G$ with maximum degree $6$ in which the vertices of degree $6$ form an independent set can be decomposed into at most $\lfloor \frac{n}{2} \rfloor$ paths, unless $G$ is $K_3$, $K_5$ or $K_5-e$. 
\\

\noindent \textbf{Our result:}
A graph $G$ is said to be \emph{k-degenerate} if every subgraph of $G$ has a vertex of degree at most $k$. In this paper, we study the path decomposition in 2-degenerate graphs. The class of 2-degenerate graphs properly include outer-planar graphs, series-parallel graphs and planar graphs of girth at least 5 as subclasses. We prove that:

\begin{theorem} \label{theorem:main}
Let $G$ be a connected 2-degenerate graph on $n$ vertices. Then the edges of $G$ can be decomposed into at most $\lfloor \frac{n}{2} \rfloor$ paths unless $G$ is a triangle. 
\end{theorem}

Note that any triangle requires two paths. We can even extend the result to 2-degenerate graphs which are not connected, but no component being a triangle. We have the following corollary:

\begin{corollary}\label{corollary:maincorollary}
Let $G$ be a 2-degenerate graph on $n$ vertices such that none of the components of $G$ is a triangle. Then the edges of $G$ can be decomposed into at most $\lfloor \frac{n}{2} \rfloor$ paths.
\end{corollary}

\section{Preliminaries}\label{sec:Preliminaries}
The degree of a vertex $v$ in a graph $G=(V,E)$ is denoted by $d_G(v)$. A vertex of degree $0$ is said to be \emph{isolated}. The set of neighbours of a vertex $v$ in $G$ is denoted by $N_G(v)$. A \emph{walk} is a non-empty alternating sequence $v_0,e_0,v_1,e_1, ...,e_{k-1},v_k$ of vertices and edges in a graph, such that any edge in the sequence has the vertex immediately before and after it as its endpoints in the graph. A \emph{path} is a walk in which all vertices and edges are distinct. The number of edges in a path $P$ is called its \emph{length} and is denoted as $|P|$. The notation $v_iPv_j$ denotes the subpath of the path $P$ starting at the vertex $v_i$ and ending at the vertex $v_j$. For a path $P$ and an edge $xy \notin P$ and $xy$ incident on one of the endpoints of $P$, $P \cup xy$ denotes the extended path obtained by adding $xy$ to $P$. A \emph{closed walk} is a walk that starts and ends at the same vertex. A \emph{cycle} is a closed walk where all the edges are distinct, and the only repeated vertices are the first and the last vertices. The number of edges in a cycle $C$ is called its \emph{length} and is denoted as $|C|$. For an edge $xy \in E$, where $x,y \in V$, $G-xy$ will denote the graph obtained by deletion of the edge $xy$. Similarly, if $F$ denotes any subset of edges of the edge set $E$, then $G - F$ denotes the graph obtained after the removal of all the edges of $F$ from $G$. For a vertex $x \in V$, $G-x$ will denote the graph obtained by deletion of the vertex $x$ and all the edges incident on $x$. Similarly, the graph $ G - \{x_0,x_1,..., x_i\}$ denotes the graph obtained if multiple vertices $x_0, x_1,..., x_i$  and the edges incident on them are removed from $G$. An \emph{odd component} is a component of a graph with an odd number of vertices. Similarly, an \emph{even component} is a component of a graph with an even number of vertices. We say that a graph $G$ on $n$ vertices has a \emph{valid decomposition} if there exists a decomposition of the edges of $G$ into at most $\lfloor \frac{n}{2} \rfloor$ paths. For any missing definitions, please refer \citet{diestel2005graph}.

\begin{definition}[Vertex removal order]\label{definition:vertexorder}
Let  $x_1, x_2, ..., x_n$ be an ordering of vertices of $G$. We denote by $H_i$, the graph obtained by removing vertices  $x_1, ..., x_{i-1}$ from $G$. If $G$ is a 2-degenerate graph on $n$ vertices, then there exists an ordering $x_1, x_2, ..., x_n$ of vertices such that $d_{H_i}(x_i) \leq 2$, for every $ i$, $  1 \leq i \leq n $
\end{definition}

\section{Proof of Theorem 1}\label{sec:Proof}

Let $G$ be a minimum counterexample to the theorem statement with respect to the  number of vertices. It is easy to see that $n > 3$ since any connected graph with at most three vertices is either a path or a triangle, and the theorem  holds in those cases.

\begin{observation}\label{observation:subgraph}
Let $H$ be a subgraph of $G$ with $n - i$ non-isolated vertices, for $i >0$, and suppose that none of the components in $H$ is a triangle. Since $G$ is a minimum counterexample, the edges of $H$ can be decomposed into at most $\lfloor \frac{n - i}{2} \rfloor$ paths.
\end{observation}

\begin{lemma}\label{lemma:triangles}
Let $P$ be a path in the graph $G$. Let $T_1$, $T_2$,..., $T_j$,  for  $j \geq 0$, be the components of $G - E(P)$ that are triangles. Then the edges of $P$ and the components $T_1$, $T_2$,..., $T_j$ can be decomposed into $j+1$ paths.
\end{lemma}
\begin{proof}
We prove the  statement using induction on the number of triangle components. Clearly, if $j=0$, the statement holds. Suppose the statement holds for $j-1$ triangles. Now we prove that the statement holds if $G-E(P)$ has $j$ triangle components $T_1, T_2,..., T_j$. Let the endpoints of the path $P$ be vertices $a$ and $b$. While travelling along $P$, starting at $a$, let $x$ be the vertex in $P$ where a triangle first intersects $P$. Without loss of generality, let $T_j$ be this triangle. Let $x$, $y$ and $z$ be the vertices of $T_j$. The triangle $T_j$ may intersect $P$ at one, two or three vertices. This gives us the following three cases. In all these cases, we decompose the edges of $P \cup T_j$ into two paths $Q$ and $R$, each of which will be used later. 
\begin{case}\label{case:lc1} $P$ intersects $T_j$ at three vertices.
\end{case}
Let the vertices of $T_j$ intersect the path $P$ in the order $x$, $y$, $z$.  See Figure~\ref{fig:triangles3}. We can define two paths $Q$ and $R$, such that $Q = aPx \cup xy \cup yz \cup zw$, where $w$ is a neighbour of $z$ and $zw$ $\in$ $zPy$ and $R = wPy \cup yPx \cup xz \cup zPb$.
\\
\begin{figure}[h]
\centering
\begin{subfigure}[b]{0.41\textwidth}\centering
\begin{tikzpicture}[scale = 0.4, node distance={15mm}, main/.style = {draw, circle}]
\node[main] (1) {$a$};
\node[main] (2) [below of=1] {$x$};
\node[main] (3) [below left of=2] {$y$};
\node[main] (4) [below right of=2] {$z$};
\node[main] (5) [ below right of=4] {$b$};
\node[main] (6) [ below left of=4] {$w$};
\draw (2) -- (3);
\draw (2) -- (4);
\draw (3) -- (4);
\draw [dashed] (2) to [out=75,in=130,looseness=1] (3);
\draw [dashed] (1) to [out=270,in=50,looseness=1.5] (2);
\draw [dashed] (3) to [ out=270,in=134,looseness=1.5] (6);
\draw [dashed] (4) to [ out=270,in=234,looseness=1.5] (5);
\draw  [dashed](4) to (6);
\end{tikzpicture}\caption{} \label{fig:triangles3va} \end{subfigure}\hspace{0.5cm}\begin{subfigure}[b]{0.41\textwidth}\centering\begin{tikzpicture}[scale = 0.6, node distance={15mm}, main/.style = {draw, circle}]
\node[main] (1) {$a$};
\node[main] (2) [below of=1] {$x$};
\node[main] (3) [below left of=2] {$y$};
\node[main] (4) [below right of=2] {$z$};
\node[main] (5) [ below right of=4] {$b$};
\node[main] (6) [ below left of=4] {$w$};
\draw [thin] (2) -- (3);
\draw [thick, line width=2.5pt] (2) -- (4);
\draw [thin] (3) -- (4);
\draw [ thick, line width=2.5pt] (2) to [out=75,in=130,looseness=1] (3);
\draw [ thin] (1) to [out=270,in=50,looseness=1.5] (2);
\draw [ thick, line width=2.5pt] (3) to [ out=270,in=134,looseness=1.5] (6);
\draw [thick, line width=2.5pt] (4) to [ out=270,in=234,looseness=1.5] (5);
\draw [ thin] (4) to (6);
\end{tikzpicture}\caption{} \label{fig:triangles3vb}
\end{subfigure}
\caption{ Three vertices of triangle $T_j$  intersect $P$. (a) The path $P$ is shown with a dashed line, while the edges of $T_j$ are shown in solid lines. (b) Decomposition of the edges of $T_j$ and $P$ into two paths $Q$ (thin line) and $R$ (bold line). } \label{fig:triangles3}
\end{figure}
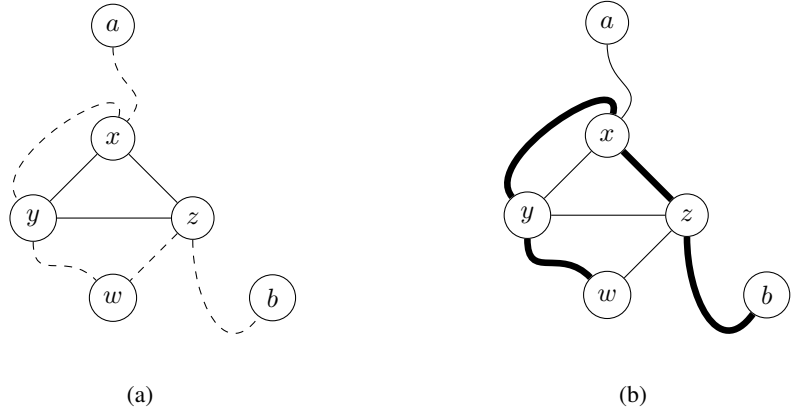
\begin{case}\label{case:lc2}  $P$ intersects $T_j$ at two vertices.
\end{case}
Let the vertices of $T_j$ intersect the path $P$ in the order $x$, $y$. See Figure~\ref{fig:triangles2}. We can define two paths $Q$ and $R$, such that $Q= aPx \cup xy \cup yw$, where $w$ is a neighbour of $y$ and $yw \in xPy$ and  $R = wPx \cup xz \cup zy \cup yPb$. 
\begin{figure}[h]
\centering
\begin{subfigure}[b]{0.41\textwidth}\centering
\begin{tikzpicture}[scale = 0.4, node distance={15mm}, main/.style = {draw, circle}]
\node[main] (1) {$a$};
\node[main] (2) [below of=1] {$x$};
\node[main] (3) [below left of=2] {$y$};
\node[main] (4) [below right of=2] {$z$};
\node[main] (5) [ below right of=3] {$b$};
\node[main] (6) [ above left of=3] {$w$};
\draw (2) -- (3);
\draw (2) -- (4);
\draw (3) -- (4);
\draw [dashed] (2) to [out=175,in=60,looseness=1] (6);
\draw [dashed] (1) to [out=270,in=50,looseness=1.5] (2);
\draw [dashed] (3) to [ out=270,in=234,looseness=1.5] (5);
\draw [dashed](3) to (6);
\end{tikzpicture}\caption{} \label{fig:triangles2va} \end{subfigure}\hspace{0.5cm}\begin{subfigure}[b]{0.41\textwidth}\centering\begin{tikzpicture}[scale = 0.6, node distance={15mm}, main/.style = {draw, circle}]
\node[main] (1) {$a$};
\node[main] (2) [below of=1] {$x$};
\node[main] (3) [below left of=2] {$y$};
\node[main] (4) [below right of=2] {$z$};
\node[main] (5) [ below right of=3] {$b$};
\node[main] (6) [ above left of=3] {$w$};
\draw [thin] (2) -- (3);
\draw [thick, line width=2.5pt](2) -- (4);
\draw [thick, line width=2.5pt] (3) -- (4);
\draw [thick, line width=2.5pt] (2) to [out=175,in=60,looseness=1] (6);
\draw [thin] (1) to [out=270,in=50,looseness=1.5] (2);
\draw [thick, line width=2.5pt] (3) to [ out=270,in=234,looseness=1.5] (5);
\draw [thin ] (3) to (6);
\end{tikzpicture}\caption{} \label{fig:triangles2vb}
\end{subfigure}
\caption{ Two vertices of triangle $T_j$  intersect $P$. (a) The path $P$ is shown with a dashed line, while the edges of $T_j$ are shown in solid lines. (b) Decomposition of the edges of $T_j$ and $P$ into two paths $Q$ (thin line) and $R$ (bold line). } \label{fig:triangles2}
\end{figure}
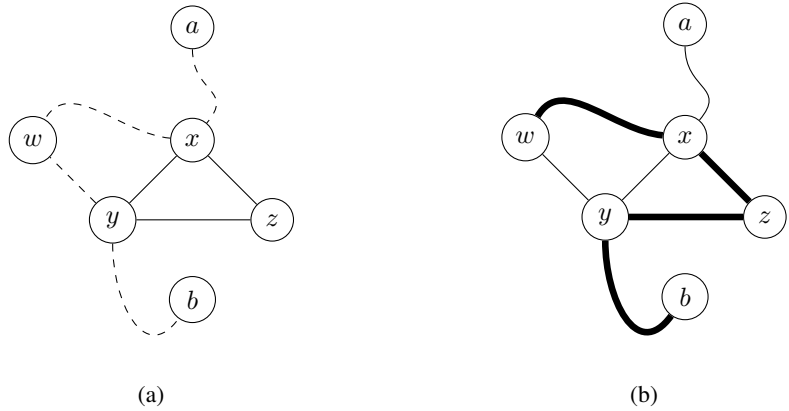
\begin{case}\label{case:lc3}  $P$ intersects $T_j$ at one vertex.
\end{case}
Let the vertex of $T_j$ that intersects the path  $P$ be $x$. See Figure~\ref{fig:triangles1}. We can define two paths $Q$ and $R$ such that $Q = aPx \cup xy \cup yz$ and $R = zx \cup xPb$.  

\begin{figure}[h]
\centering
\begin{subfigure}[b]{0.41\textwidth}\centering
\begin{tikzpicture}[scale = 0.4, node distance={15mm}, main/.style = {draw, circle}]
\node[main] (1) {$a$};
\node[main] (2) [below of=1] {$x$};
\node[main] (3) [below left of=2] {$y$};
\node[main] (4) [below right of=2] {$z$};
\node[main] (5) [ right of=2] {$b$};
\draw (2) -- (3);
\draw (2) -- (4);
\draw (3) -- (4);
\draw [dashed] (2) to [out=10,in=170,looseness=1] (5);
\draw [dashed] (1) to [out=270,in=50,looseness=1.5] (2);
\end{tikzpicture}\caption{} \label{fig:triangles1va} \end{subfigure}\hspace{0.5cm}\begin{subfigure}[b]{0.41\textwidth}\centering\begin{tikzpicture}[scale = 0.6, node distance={15mm}, main/.style = {draw, circle}]
\node[main] (1) {$a$};
\node[main] (2) [below of=1] {$x$};
\node[main] (3) [below left of=2] {$y$};
\node[main] (4) [below right of=2] {$z$};
\node[main] (5) [ right of=2] {$b$};
\draw [thin] (2) -- (3);
\draw [thick, line width=2.5pt](2) -- (4);
\draw [thin] (3) -- (4);
\draw [thick, line width=2.5pt] (2) to [out=10,in=170,looseness=1] (5);
\draw [thin] (1) to [out=270,in=50,looseness=1.5] (2);
\end{tikzpicture}\caption{} \label{fig:triangles1vb}
\end{subfigure}
\caption{ One vertex of triangle $T_j$  intersects $P$. (a) The path $P$ is shown with a dashed line, while the edges of $T_j$ are shown in solid lines. (b) Decomposition of the edges of $T_j$ and $P$ into two paths $Q$ (thin line) and $R$ (bold line). } \label{fig:triangles1}
\end{figure}
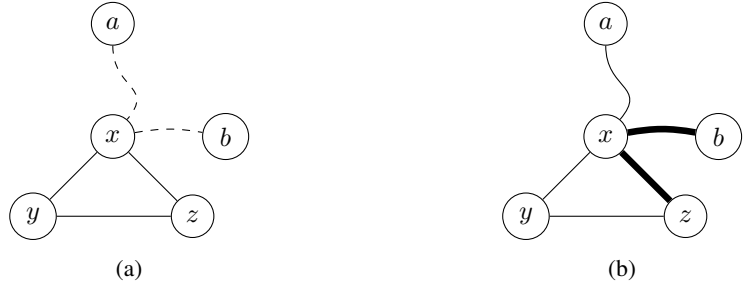

Let $G' = G- E(Q)$. It is easy to see that $G'$ does not have any triangle components using the following argument: since $T_j$ is the first triangle to intersect $P$, the removal of $aPx$ does not produce any other triangle components. The removal of the rest of the edges in the path $Q$ does not create any new triangle components either, because in Case~\ref{case:lc1} and Case \ref{case:lc2}, the vertices $x$,$y$, $z$ and $w$ are still connected to the path $R$ and in Case \ref{case:lc3}, $x$ and $z$ are still connected to the path $R$, while $y$ becomes isolated. \\ \\
Path $Q$ does not contain edges from triangles $T_1, T_2, ..., T_{j-1}$. Let $G'' = G'- E(R) = G - E(Q \cup R)$. Note that $G'' = G - E(P \cup T_{j})$.  Therefore, the removal of the path $R$ from $G'$ will create $j-1$ triangle components $T_1, T_2, ..., T_{j-1}$. By inductive hypothesis, the edges of $R$ and the edges of  $T_1, T_2, ..., T_{j-1}$ can be decomposed into  $j$ paths. These $j$ paths, along with the path $Q$, give us the required $j+1$ paths.
 \end{proof}
\begin{claim_new}\label{lemma:deg12}
$G$ does not contain two vertices with degree at most 2.

\end{claim_new}

\begin{proof}
Let $u$ and $v$ be the closest pair of vertices in $G$ with degrees at most $2$. Our proof relies on removing the edges of a path or a cycle that isolates the vertices $u$ and $v$. In each case  of the proof that follows, we will attempt to get a valid decomposition of the edges of $G$, using a valid decomposition of the edges of a subgraph of $G$ having at most $n-2$ non-isolated vertices. \\

Let $P'$ be a shortest path connecting vertices $u$ and $v$. Suppose $|P'| >2$. If $d_G(u) =1$ and $d_G(v) = 1$, then let $P = P'$. Otherwise, if either or both of $u$ and $v$ have degree 2, then extend the path $P'$ to include all the edges incident on $u$ and $v$. Let $P$ be this extended path. Removal of the edges of $P$ from $G$ removes all the edges incident on $u$ and $v$, making them isolated.\\ \\
On the other hand, suppose $|P'| \leq 2$. If $d_G(u) =1$ and $d_G(v) = 1$, let $P = P'$. If either or both of $u$ and $v$ have degree 2, then extend the path $P'$ to include all the edges incident on $u$ and $v$ to get a walk $L$. Since $|P'| \leq 2$, we can infer that $|L| \leq 4$ and no edge is repeated. Thus the walk $L$ is a path or a cycle. If the walk $L$ is a path, let $P = L$. Otherwise, it is a cycle of length either $3$ or $4$. Note that $u$ and $v$ are isolated vertices in $G- E(L)$.
\\ \\
Now consider the case when $|P'| > 2$ or $|P'| \leq 2$ and $L$ is a path.  The path $P$ is such that $G - E(P) $ has at least two isolated vertices $u$ and $v$. Let  $j \geq 0$ be the number of triangle components of $G-E(P)$. Let $H$ denote the graph obtained from $G-E(P)$, after the $j$ triangle components are removed. $H$ has at most $n - 2 - 3j$ non-isolated vertices. Since the graph $G$ is a minimum counterexample, by Observation~\ref{observation:subgraph}, the edges in $H$ can be decomposed into at most $\lfloor \frac{n - 2 - 3j}{2} \rfloor$ paths. By Lemma~\ref{lemma:triangles}, the edges of the $j$ triangle components together with $P$ require $j+1$ paths. This implies that the edges of $G$ can be decomposed into at most $\lfloor \frac{n - 2 - 3j}{2} \rfloor+ j+1 \leq \lfloor \frac{n}{2} \rfloor$ paths.\\ \\ 
Now we are left with the case $|P'| \leq 2$ and $L$ is a cycle, which we denote as $C$.
\begin{subclaim}\label{claim:atmost1_cycle}
If the edges of $ G-E(C)$ can be decomposed into $k$ paths, then the edges of $G$ can be decomposed into at most $k+1$ paths.
\end{subclaim}

\begin{proof}
As mentioned before, the length of $C$ is either $3$ or $4$. We will consider these cases separately. Let $W$ be a path in the path decomposition of $G - E(C)$ that intersects $C$ in $G$. Let vertices $a$ and $b$ be the two endpoints of this path $W$. We will demonstrate that the edges of $W$ and $C$ together can be decomposed into at most $2$ paths in $G$. If the length of cycle $C$ is $4$, let $x$ and $y$ be the common neighbours of $u$ and $v$. The path $W$ can intersect $C$ in $G$ at either one or both of $x$ and $y$. If the path $W$ intersects $C$ at both $x$ and $y$ (in that order), we can define two paths $W_1$ and $W_2$, such that $W_1 = aWx \cup xv \cup vy \cup yu$ and $W_2 = ux \cup xWy \cup yWb$. If the path $W$ intersects $C$ at only one vertex (let this be $x$), we can define two paths $W_1$ and $W_2$ such that $W_1 = aWx \cup xv \cup vy \cup yu$ and $W_2 = ux \cup xWb$. \\ \\ 
If the cycle $C$ is of length 3, $W$ can intersect $C$ at exactly one vertex (let this be $x$). We can define two paths $W_1 = aWx \cup xv \cup vu $ and $W_2 = ux \cup xWb$. Therefore, the edges of $W$ and $C$ together can be decomposed into at most two paths in $G$. In other words, if the edges of $G - E(C)$ can be decomposed into $k$ paths, then the edges of $G$ can be decomposed into at most $k+1$ paths.  
\end{proof}

\begin{subclaim}\label{claim:cycle_triangle}
Let $T$ be a triangle component of $G - E(C)$. Then, the edges of $T$ and $C$ together can be decomposed into at most two paths.
\end{subclaim}

\begin{proof}
As mentioned earlier, $C$ is either of length $3$ or $4$. The triangle $T$ and the cycle $C$ can have either one or two vertices in common. If $T$ and $C$ have one vertex in common, then there exist two adjacent vertices of degree $2$ in $T$ and hence, in $G$. By definition, $u$ and $v$ are the closest pair of vertices in $G$, whose degrees are at most $2$. Therefore, $u$ and $v$ also have to be adjacent. So, $C$ can only be of length $3$ here. Hence, if at most one vertex in $T$ intersects $C$, the resulting graph is a graph with two triangles intersecting at a common vertex, say $x$. It is easy to see that the edges in this graph can be decomposed into at most two paths. If $T$ and $C$ have two common vertices, the length of $C$ has to be $4$, because $u$ and $v$ cannot be part of $T$. Let the common vertices in $C$ and $T$ be $x$ and $y$. Let $a$ be the third vertex in $T$. It is easy to see that the edges of $C$ and $T$ together require at most two paths : namely, $P_1 = vy \cup ya \cup ax \cup xu$ and $P_2 = uy \cup yx \cup xv$.  
\end{proof}

If $G - E(C)$ has no triangle component, then by Subclaim~\ref{claim:atmost1_cycle} and Observation~\ref{observation:subgraph}, the edges of $G$  can be decomposed into at most $\lfloor \frac{n}{2} \rfloor$ paths. Suppose $T$ is a triangle component of $G-E(C)$. Clearly, at most one triangle component is present in $G -E(C)$ when $|C| =3$ and that triangle component will contain the vertex $x$. Because $u$ and $v$ are the closest pair of vertices with degree at most $2$ in $G$, at most one triangle component is present in $G -E(C)$ when $|C| = 4$ and that triangle component will contain both the vertices $x$ and $y$. By Subclaim~\ref{claim:cycle_triangle}, the edges of $C$ and $T$  can be decomposed into at most $2$ paths. This implies that the edges of $G$ can be decomposed into at most $\lfloor \frac{n - 2 - 3}{2} \rfloor+ 2 \leq \lfloor \frac{n}{2} \rfloor$ paths.  
\end{proof}

\begin{observation}\label{observation:floor} Let $H$ be a graph on $n$ vertices that consists of two odd components $I$ and $J$ with $n_I$ and $n_J$ vertices respectively, with $n_I + n_J = n$, and suppose none of these components is a triangle. Suppose the  edges of $I$ and $J$ admit a path decomposition with at most $\lfloor \frac{n_I }{2} \rfloor$ and $\lfloor \frac{n_J }{2} \rfloor$ paths respectively. Then the edges of $H$ can be decomposed into at most $\lfloor \frac{n_I + n_J}{2} \rfloor -1 $ paths, which is one path less than the bound given by the conjecture for a graph with $n_I + n_J$ vertices.

\end{observation}
\begin{figure}[h]
\centering
\begin{subfigure}[b]{0.3\textwidth}

\resizebox{\linewidth}{!}{
\begin{tikzpicture} [node distance={15mm}, main/.style = {draw, circle}]
\node[main] (1) {$v$};
\node[main] (2) [ right of=1] {$x$};
\node[main] (3) [above right of=2] {$w$};
\node[main,shape = rectangle,minimum size=0.4 cm] (4) [below right of=2] {$z$};
\node[main,shape = rectangle,minimum size=0.4 cm] (5) [above right of=3] {$a$};
\node[main,shape = rectangle,minimum size=0.4 cm] (6) [ right of=3] {$b$};
\draw (1) -- (2);
\draw (2) -- (3);
\draw (2) -- (4);
\draw (3) -- (5);
\draw (6) -- (3);
\end{tikzpicture}
}
\caption{} \label{fig:pendant}
\end{subfigure}%
\begin{subfigure}[b]{0.5\textwidth}
\resizebox{\linewidth}{!}{
\begin{tikzpicture}[node distance={15mm}, main/.style = {draw, circle}]
\node[main] (1) {$v$};
\node[main] (2) [below right of=1] {$x$};
\node[main,shape = rectangle,minimum size=0.4 cm] (3) [below left of=1] {$y$};
\node[main,shape = rectangle,minimum size=0.4 cm] (4) [below of=1] {$z$};
\node[main,shape = rectangle,minimum size=0.4 cm] (5) [ right of=2] {$w$};
\node[main, draw=none] (6) [ right of=5] {};
\node[main,draw=none] (7) [ left of=3] {};
\draw (1) -- (2);
\draw (2) -- (5);
\draw (2) -- (4);
\draw (1) -- (3);
\end{tikzpicture}
}
\caption{} \label{fig:xcut}
\end{subfigure}
\caption { (a)  Claim~\ref{lemma:deg1} : Vertex $v$ cannot be a pendant vertex. (b)  Claim~\ref{lemma:deg3cut} : Vertex $x$ cannot be a cut vertex.  Any vertex which is drawn as a circle has all its edges depicted in the figure. Rectangular vertices may have edges not depicted in the figure. 
}
\end{figure}
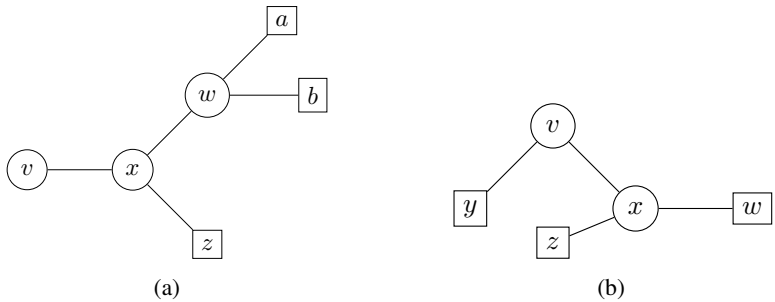
By Claim~\ref{lemma:deg12}, $G$ contains exactly one vertex of degree at most 2. Let this vertex be $v$. Let $N_G(v) = \{x\}$ if $d_G(v) = 1$ and $N_G(v) = \{x, y\}$ if $d_G(v) = 2$ . 
\\ \\ Now if we consider the 2-degenerate ordering of vertices of $G$ as given in Definition~\ref{definition:vertexorder}, the vertex $v$ will be the first vertex in the ordering. By claim~\ref{lemma:deg12}, all other vertices have degree at least $3$ in $G$. Note that $ G - v$ is again 2-degenerate and hence contains a vertex of degree at most $2$. We can infer that the vertex with degree at most $2$ in $G -v $ has to be either $x$ or $y$. Without loss of generality, let $x$ be that vertex. Thus $d_G(x) =3 $. Let $N_G(x) = \{v, w, z\}$.

\begin{claim_new}\label{lemma:deg1}
$G$ does not contain a pendant vertex.
\end{claim_new}
\begin{proof}
Suppose $d_G(v) = 1$. Since $G - \{v,x\}$ is also a 2-degenerate graph, it is easy to see that at least one of $w$ or $z$ in $G$ must have degree 3. Without loss of generality, let it be $w$ (See Fig~\ref{fig:pendant}).
\setcounter{case}{0}
\begin{case}
$x$ is not a cut vertex in $G - v$. 
\end{case}
Let $G' = G - \{vx, xw\}$. If removing edges $vx$ and $xw$ produced a triangle component, then $G$ contains at least two vertices with degree at most $2$ and violates Claim~\ref{lemma:deg12}. Since $x$ is not a cut vertex in $G - v$, there exists a path from $w$ to $z$, which does not go through any edges incident on $x$. Let $P'$ be a shortest path between $w$ and $z$ that does not contain $x$. Let $P = P' \cup xz$. Let $G'' = G' - E(P)$. Let $j \geq 0$ be the number of triangle components in $G''$. Let $H$ denote the graph obtained from $G''$, after the $j$ triangle components are removed. In $H$, $x$ and $v$ are isolated. By Observation~\ref{observation:subgraph}, the edges in $H$ can be decomposed into at most $\lfloor \frac{n - 2 - 3j}{2} \rfloor$ paths. The vertex $w$ has degree 1 in $H$. The path ending at $w$ in $H$ can be extended to $v$, by adding edges $xw$ and $vx$. By  Lemma~\ref{lemma:triangles},  the $j$ triangle components together with $P$ require $j+1$ paths. Then, the edges of $G$ can be decomposed into at most $\lfloor \frac{n - 2 - 3j}{2} \rfloor+ j+1 \leq \lfloor \frac{n}{2} \rfloor$ paths.
\begin{case}
$x$ is a cut vertex in $G- v$.
\end{case}
Since $x$ is a cut vertex in $G - v$, removal of $x$ will create two components, say $I$ and $J$. Let $J$ be the component containing vertex $w$ and $I$ be the component containing vertex $z$. Since the components $I$ and $J$ are 2-degenerate, they must have a vertex with degree at most $2$. By Claim~\ref{lemma:deg12}, we can see that such a vertex can not have degree $\leq$ 2 in $G$. Therefore, we can conclude that $d_G(z) = d_G(w) = 3$. \\ \\ 
Suppose both $J$ and $I$ have an odd number of vertices. By Observation~\ref{observation:floor}, the edges of components $J$ and $I$ together can be decomposed into at most $\lfloor \frac{n - 2}{2} \rfloor -1$ paths (note that $ G - \{v, x \}$ can not have any triangle components. Otherwise, $G$ contains at least two vertices with degree at most $2$ and violates Claim~\ref{lemma:deg12}). So we can use up to two paths to decompose the remaining edge set of $\{vx, xw, xz\}$, without going over the allowed limit of $\lfloor \frac{n }{2} \rfloor$ paths for $G$. Now if both $I$ and $J$ do not have an odd number of vertices, then at least one of the components has an even number of vertices. Since $d_G(z) = d_G(w) = 3$,  without loss of generality, let the component $J$ containing vertex $w$ have an even number of vertices. Here the proof splits into two cases again.
\begin{subcase}
$w$ is a cut vertex in $J$.
\end{subcase}
Since $w$ is a cut vertex, the removal of $w$ in $J$ creates two components $J_1$ and $J_2$. Since $J$ has an even number of vertices, one of the components in $J$ (after removal of $w$) has an odd number of vertices and the other component has an even number of vertices. Let $J_1$ be the odd component and $J_2$ be the even component. Let $a$ be the neighbour of $w$ in $J_1$, $b$ be its neighbour in $J_2$. The graph induced by $J_2$ and $w$ is an odd component. Therefore, by Observation~\ref{observation:floor}, the edges of component $J$ (excluding the edge $aw$) require ($\lfloor \frac{n_J }{2} \rfloor -1) $ paths,  where $n_J$ is the number of vertices in $J$.  Note that the edges inside the component $I$, regardless of whether it has odd or even number of vertices, can be decomposed into at most $\lfloor \frac{n_I }{2} \rfloor$ paths by Observation~\ref{observation:subgraph}, where $n_I$ is the number of vertices in $I$. We need one path for edges $aw$, $xw$, and $xz$. The edge $vx$ requires another path in the path decomposition of $G$. So, the edges of $G$ can be decomposed into at most $(\lfloor \frac{n_J }{2} \rfloor -1) + 1+ \lfloor \frac{n_I}{2} \rfloor +1 $ $\leq$ $\lfloor \frac{n }{2} \rfloor$ paths.
\begin{subcase}
$w$ is not a cut vertex in $J$.
\end{subcase}
Since the graph $J$ is also 2-degenerate, and because of Claim~\ref{lemma:deg12}, at least one of the vertices adjacent to $w$ must be a degree $3$ vertex. Let it be vertex $a$. Let $G'= G - \{vx,xw,wa\}$. Now we show that $G'$ has no triangle component. Let $Q$ denote the path $vx \cup xw \cup wa$. Note that $ G' = G - E(Q)$. Let us assume that a triangle  component $T$ is present in $G'$. Clearly, $T$ must intersect with $Q$ at more than two vertices. Otherwise, $T$ contains at least one vertex with degree $2$ in $G$, which violates Claim~\ref{lemma:deg12}. Therefore all the vertices of $T$ must intersect $Q$ and they must be non-adjacent in $Q$. But that is not possible for the path $Q$ containing only three edges. Therefore, no triangle component exists in $G'$.\\ \\
Since $w$ is not a cut vertex, there exists a path between its neighbours $a$ and $b$ that does not go through any of the edges incident on $w$. Let $P'$  be a shortest such path. Let $P = P' \cup bw$. Let $G'' = G' - E(P)$. Let $k \geq 0$ be the number of triangle components of $G''$. Let $H$ denote the graph obtained from $G''$, after the $k$ triangle components are removed. In $H$, $w$ and $v$ are isolated. By Observation~\ref{observation:subgraph}, edges in $H$ can be decomposed into at most $\lfloor \frac{n - 2 - 3k}{2} \rfloor$ paths. By Lemma~\ref{lemma:triangles}, the $k$ triangle components together with $P$ require $k+1$ paths. The vertex $a$ has degree $1$ in $H$. The path ending at $a$ can be extended till $v$ without needing any additional paths (note that extending the path to $v$ does not create a cycle because $v$ and $w$ are isolated vertices in $H$ and the vertex $x$ is a cut vertex by our assumption). This takes care of edges $aw$,$wx$,$xv$. Then, the edges of $G$ can be decomposed into at most $\lfloor \frac{n - 2 - 3k}{2} \rfloor+ k+1 \leq$ $ \lfloor \frac{n}{2} \rfloor$ paths.     
\end{proof}

From Claim~\ref{lemma:deg12} and Claim~\ref{lemma:deg1}, we can conclude that $v$ is the only vertex in $G$ with degree at most 2, and that $d_G(v) = 2$.
\\
\begin{claim_new}\label{lemma:deg2cut}
$G$ does not contain a degree 2 cut vertex.
\end{claim_new}

\begin{proof}
Suppose $v$ is a cut vertex in $G$, whose removal separates the graph into components $X$ and $Y$ with $n_X$ and $n_Y$ vertices respectively. Let $x$ and $y$ be the neighbours of $v$ in components $X$ and $Y$ respectively. Consider the graph $G - v$. Let component $X' = X + vx $, with $n_X+1$ vertices. If components $X'$ and $ Y$ are not triangles, then by Observation~\ref{observation:subgraph}, the edges of $X'$ and $ Y$ can be decomposed into $\lfloor \frac{n_X +1 }{2} \rfloor$ paths and $\lfloor \frac{n_Y}{2} \rfloor$ paths respectively. The path in $X'$ that ends at $v$ can be extended by adding back the edge $vy$. This does not create any new paths. Clearly, component $X'$ is not a triangle. By Claim~\ref{lemma:deg12}, component $Y$ cannot be a triangle either. Hence, the edges of $G$ can be decomposed into at most $\lfloor \frac{n_X +1 }{2} \rfloor + \lfloor \frac{n_Y}{2} \rfloor$ $\leq$ $\lfloor \frac{n }{2} \rfloor$ paths, which is a contradiction. This implies that $v$ cannot be a cut vertex in $G$.  
\end{proof}

\begin{claim_new}\label{lemma:deg3cut} The vertex $x$ cannot be a cut vertex in $G$.
\end{claim_new}
\begin{proof}
Suppose $x$ is a cut vertex in $G$. Since $v$ can not be  a cut vertex, there exists a path between its neighbours $x$ and $y$, that does not use any of the edges incident on $v$ (See Fig~\ref{fig:xcut}). Let $z$ be the neighbour of $x$ along this path (note that $z$ can be $y$ itself, but that will not affect our proof). Let $G' = G - \{xz, vx\}$. Then, $G'$ has two components by our assumption that $x$ is a cut vertex. Clearly, neither components are triangle component. Let $n_A$ and $n_B$ denote the number of vertices in the two components. Since $G$ is a minimum counterexample, by Observation~\ref{observation:subgraph}, we can decompose the edges of the components into at most $\lfloor \frac{n_A}{2} \rfloor$ and $\lfloor \frac{n_B}{2} \rfloor$ paths respectively. Note that vertices $v$, $y$ and $z$ are in the same component, while vertex $x$ is in a different component. Since vertex $v$ and vertex $x$ each have degree $1$ in their respective components, we can add the removed edges $vx$ and $xz$ without creating any new paths. This can be done by adding the edge $vx$ to the path ending at $v$, and the edge $xz$ to the path ending at $x$, in their respective components. So, the edges of $G$ can be decomposed into at most $\lfloor \frac{n_A}{2} \rfloor$ $+$ $\lfloor \frac{n_B}{2} \rfloor$ $\leq$ $\lfloor \frac{n }{2} \rfloor$ paths, which is a contradiction.
\end{proof}
 
By Claim~\ref{lemma:deg3cut}, $x$ is not a cut vertex. Recall that $N_G(x) = \{v,z,w\}$. Based on the degrees of neighbours of $x$ we have the following cases: 
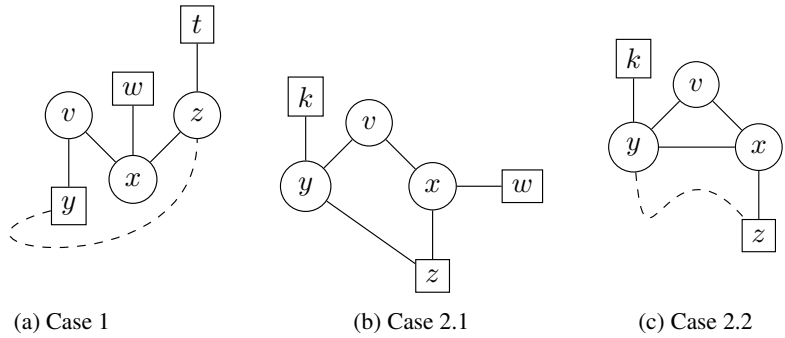
\begin{figure}[h!]
\centering
\begin{subfigure}[b]{0.3 \textwidth}
\resizebox{\linewidth}{!}{
\begin{tikzpicture} [node distance={11mm}, main/.style = {draw, circle}]
\node[main] (1) {$v$};
\node[main] (2) [below right of=1] {$x$};
\node[main,shape = rectangle,minimum size=0.4 cm] (3) [below of=1] {$y$};
\node[main,shape = rectangle,minimum size=0.4 cm] (4) [ above of=2] {$w$};
\node[main]  (5) [ above right of=2] {$z$};
\node[main,shape = rectangle,minimum size=0.4 cm] (6) [ above of=5] {$t$};
\draw (1) -- (2);
\draw (2) -- (5);
\draw (2) -- (4);
\draw (1) -- (3);
\draw (5) -- (6);
\draw [dashed] (3) to [ out=195,in=270,looseness=2.4] (5);
\end{tikzpicture}
}\caption{Case \ref{ncv1}  } \label{fig:casea} 
\end{subfigure} \hspace{0.5cm}\begin{subfigure}[b]{0.25 \textwidth}
\resizebox{\linewidth}{!}{
\begin{tikzpicture}[node distance={11mm}, main/.style = {draw, circle}]
\node[main] (1) {$v$};
\node[main] (2) [below right of=1] {$x$};
\node[main] (3) [below left of=1] {$y$};
\node[main,shape = rectangle,minimum size=0.4 cm] (6) [above of=3] {$ k$};
\node[main,shape = rectangle,minimum size=0.4 cm] (4) [ right of=2] {$w$};
\node[main,shape = rectangle,minimum size=0.4 cm](5) [ below of=2] {$z$};
\draw (1) -- (2);
\draw (2) -- (5);
\draw (2) -- (4);
\draw (3) -- (5);
\draw (1) -- (3);
\draw (6) -- (3);

\end{tikzpicture}
}
\caption{ Case \ref{ncv2.1}}\label{fig:casec}
\end{subfigure} \hspace{0.5cm} \begin{subfigure}[b]{0.17\textwidth}
\resizebox{\linewidth}{!}{
\begin{tikzpicture}[node distance={11mm}, main/.style = {draw, circle}]
\node[main] (1) {$v$};
\node[main] (2) [below right of=1] {$x$};
\node[main](3) [below left of=1] {$y$};
\node[main,shape = rectangle,minimum size=0.4 cm] (4) [ above of=3] {$k$};
\node[main,shape = rectangle,minimum size=0.4 cm] (5) [ below of=2] {$z$};
\draw (4) -- (3);
\draw (1) -- (2);
\draw (2) -- (5);
\draw (3) -- (2);
\draw (1) -- (3);
\draw [dashed] (3) to [ out=275,in=130,looseness=2.6] (5);

\end{tikzpicture}
}
\caption{ Case \ref{ncv2.2}}\label{fig:caseb}
\end{subfigure}
\caption {Figure depicts scenarios if $x$ is not a cut vertex. Case \ref{ncv1}: $x$ has a degree $3$ neighbour $z$. Case \ref{ncv2.1}: $x$ has no degree $3$ neighbour and $d_G(y) = 3$. Case \ref{ncv2.2}: $x$ has no degree $3$ neighbour and $d_G(y) = 4$. Any vertex which is drawn as a circle has all its edges depicted in the figure. Rectangular vertices may have edges not depicted in the figure.} 
\end{figure}
\begin{case2}\label{ncv1} $x$ has a degree 3 neighbour.
 \end{case2}
Without loss of generality, let that neighbour be $z$ (See Fig~\ref{fig:casea}). Let $G'= G - \{xz\}$. Since $x$ is not a cut vertex in $G$, there exists a path from $z$ to $v$, that does not go through $x$. Let $P'$ be a shortest such path. Let $P = P' \cup zt \cup xv$, where $t$ is the other neighbour of $z$ that does not lie on $P'$.  Let $G'' = G' -E(P)$. Let $j \geq 0$ be the number of triangle components of $G''$. Note that removal of the path $P$ isolates vertices $z$ and $v$. Let $H$ denote the graph obtained from $G''$, after the $j$ triangle components are removed. The edges in $H$ can be decomposed into at most $\lfloor \frac{n - 2 - 3j}{2} \rfloor$ paths. The vertex $x$ has degree 1 in $H$. The edge $xz$ can be added back by extending the path in $H$ that ends at $x$. The $j$ triangle components together with $P$ require $j+1$ paths, by Lemma~\ref{lemma:triangles}. Then, the edges of $G$ can be decomposed into at most $\lfloor \frac{n - 2 - 3j}{2} \rfloor+ j+1$ $\leq$ $ \lfloor \frac{n}{2} \rfloor$ paths. Thus $d_G(z) \geq4 $ and $d_G(w) \geq 4$.  \\

\begin{case2} \label{ncv2} $x$ has no degree 3 neighbour.
 \end{case2} 
Based on the degree of $y$, we have following subcases:\\
\begin{subcase2} \label{ncv2.1} $d_G(y) = 3$.  \end{subcase2}

Note that the case of $y$ having a degree $3$ neighbour is similar to Case \ref{ncv1}, with $y$ playing the role of $x$. Vertices $v$, $x$ and $y$ are the first three vertices in the vertex removal order given by Definition~\ref{definition:vertexorder}. Let some vertex $a$ be the next vertex in the order after $v, x, y$. Vertex $v$ is the only vertex with degree $2$ in $G$. Since $G - \{v,x,y\}$ is also 2-degenerate, vertex $a$ has to be a neighbour of $x$ or $y$. But $x$ and $y$ have no neighbour with degree $3$. Therefore, $d_G(a) \geq 4$.  Since vertex $a$ has degree $2$ in $G - \{v,x,y\}$, $a$ must be a neighbour of both $x$ and $y$ and $d_G(a) = 4$. Without loss of generality, let $z$ be $a$.(See Fig~\ref{fig:casec}). \\ \\
Recall that $w$ is the other neighbour of $x$. Let $G' = G - \{vx, vy\}$.  Let $P =yz \cup zx \cup xw$. Let $G''=G' - E(P)$. Let $j \geq 0$ be the number of triangle components of $G''$. The vertices $x$ and $v$ are isolated now. Let $H$ denote the graph obtained from $G''$, after the $j$ triangle components are removed. The edges in $H$ can be decomposed into at most $\lfloor \frac{n - 2 - 3j}{2} \rfloor$ paths. $d_H(y) = 1$. So, the path ending at $y$ can be extended to include edges $vy$ and $vx$. The $j$ triangle components and $P$ require at most $j+1 $ paths. Then, the edges of $G$ can be decomposed into at most $\lfloor \frac{n - 2 - 3j}{2} \rfloor+ j+1$ $\leq$ $ \lfloor \frac{n}{2} \rfloor$ paths. Therefore, $d_G(y) \geq 4$.
\begin{subcase2} \label{ncv2.2} $d_G(y) \geq 4$.\end{subcase2}
Vertices $v$ and $x$ are the first two vertices in the vertex removal order given by Definition~\ref{definition:vertexorder}. Since $G - \{v,x\} $ is also 2-degenerate, the next vertex in the order must also be a neighbour of $x$ or $v$. Recall that $d_G(y) \geq 4$ and $x$ has no neighbour with degree $3$ or less. Since $G - \{v,x\}$ is 2-degenerate and must have a vertex with degree at most $2$, $d_G(y) = 4$ and $y$ has to be a neighbour of $x$ (See Fig~\ref{fig:caseb}).\\  \\
Let $G' = G - \{xy, xv\}$. If removing edges $xy$ and $xv$ created a triangle component in $G'$, then $G$ contains at least two vertices with degree at most $2$ and that violates Claim~\ref{lemma:deg12}. Since $x$ is not a cut vertex, there exists a path from $z$ to $v$ that does not go through any of the incident edges on $x$. Let $P'$ be a shortest such path. Let  $P = P' \cup zx$.  Let $G''= G' -E(P)$. Let  $j \geq 0$ be the number of triangle components of $G''$.  Let $H$ denote the graph obtained from $G''$, after the $j$ triangle components are removed. The vertices $x$ and $v$ are isolated in $H$. The vertex $y$ in $H$ is of degree $1$. So, the edges $xy$ and $xv$ can be added to the path ending at $y$ in $H$. Then, the edges of $G$ can be decomposed into at most $\lfloor \frac{n - 2 - 3j}{2} \rfloor+ j+1 \leq \lfloor \frac{n}{2} \rfloor$ paths. \\ \\
Therefore, we infer that $G$ is not a counterexample. This completes our proof for Theorem ~\ref{theorem:main}.\qed
\section{Conclusion}\label{sec:con}
We have proved Gallai's path decomposition conjecture for 2-degenerate graphs. ~\cite{botler2020gallai} proved the conjecture for graphs with treewidth at most 3 and ~\cite{botler2019gallai}  proved the conjecture for triangle-free planar graphs. Note that these graphs are proper subclasses of 3-degenerate graphs. Hence proving the conjecture for 3-degenerate graphs would generalize the results in~\cite{botler2020gallai} and~\cite{botler2019gallai}.

\acknowledgements
\label{sec:ack}
We would like to thank Ashwin Joisa, Dibyadarshan Hota and  Sai Charan for the initial discussions on the problem.

\nocite{*}
\bibliographystyle{abbrvnat}
\bibliography{references}
\label{sec:biblio}

\end{document}